\allowbreak\title{\bf Singular perturbation analysis for a coupled KdV-ODE system$^{1}$}

\author{Swann Marx$^{2}$ and Eduardo Cerpa$^{3}$}

\documentclass[12pt, draft]{article}
 \usepackage{amsfonts,latexsym,amsmath,amssymb,amsthm}
\usepackage[mathscr]{eucal}
\usepackage{graphicx,color}
\usepackage{comment}
\usepackage[hidelinks]{hyperref}
\usepackage{epstopdf}
\usepackage[margin = 2cm]{geometry}
\usepackage[inline]{enumitem}
\usepackage{mathtools}
\usepackage{cite}
\usepackage{xcolor}
\usepackage{cleveref}
\usepackage{tikz}
\usepackage{array}
\usepackage{calc}
\usepackage[normalem]{ulem}
\usepackage{stmaryrd}
\usepackage{bbm}

\usetikzlibrary{arrows, arrows.meta}

\newtheorem{theorem}{Theorem}
\newtheorem{lemma}[theorem]{Lemma}
\newtheorem{proposition}[theorem]{Proposition}

\theoremstyle{definition}

\newcommand{\suchthat}{\ifnum\currentgrouptype=16 \mathrel{}\middle|\mathrel{}\else\mid\fi}

\makeatletter
\newcommand{\customlabel}[2]{%
   \protected@write \@auxout {}{\string \newlabel {#1}{{#2}{\thepage}{#2}{#1}{}} }%
   \hypertarget{#1}{#2}
}
\makeatother

\begin{document}

\setlist[enumerate]{label={(\alph*)}}
\setlist[enumerate, 2]{label={(\alph{enumi}-\roman*)}}
\newlist{listhypo}{enumerate}{1}
\setlist[listhypo]{label={\textup{(H\arabic*)}}, ref={\textup{(H\arabic*)}}}

\footnotetext[1]{This work has been partially supported by ANID Millennium Science Initiative
Program through Millennium Nucleus for Applied Control and Inverse Problems NCN19-161 and STIC-Amsud project C-CAIT.}
\footnotetext[2]{LS2N, \'Ecole Centrale de Nantes \& CNRS UMR 6004, F-44000 Nantes, France. E-mail: swann.marx@ls2n.fr}
\footnotetext[3]{Instituto de Ingenier\'ia Matem\'atica y Computacional, Facultad de Matem\'aticas, Pontificia Universidad Católica de Chile, Avda.
Vicu\~na Mackenna 4860, Macul, Santiago, Chile. E-mail: eduardo.cerpa@uc.cl}
\maketitle
\abstract{Asymptotic stability is with no doubts an essential property to be studied for any system. This analysis often becomes very difficult for coupled systems and even harder when different time-scales appear. The singular perturbation method allows to decouple a full system into what are called the reduced order system and the boundary layer system, to get simpler stability conditions for the original system. In the infinite-dimensional setting, we do not have a general result making sure this strategy works. This papers is devoted to this analysis for some systems coupling the Korteweg-the Vries equation and an ordinary differential equation with different time-scales. More precisely, We obtain stability results and Tikhonov-type theorems.}

\vspace{0.2 cm}

{\bf Keywords:} Dispersive systems, time scales, perturbation, stability


\section{Introduction}


This paper is devoted to the stability analysis of a system composed by a Korteweg-de Vries (for short KdV) equation coupled with a scalar ordinary differential equation (ODE) with different time scales. Such a situation may appear when the control (appearing in the ODE) can only be used through a dynamics (given by the ODE), and when one of the equations is faster than the other one. More precisely, we are interested in the system

\begin{equation}
\label{eq:fast_KdV}
\left\{
    \begin{array}{cl}
        &\varepsilon y_t + y_x + y_{xxx}=0,\: (t,x)\in \mathbb R_+\times [0,L],\\
        &y(t,0) = y(t,L) = 0,\: t\in \mathbb R_+,\\
        &y_x(t,L) = a z(t),\: t\in\mathbb R_+,\\
        &y(0,x) = y_0(x),\: x\in [0,L],\\
        &\dot z(t) = b z(t) + c y_x(t,0),\: t\in \mathbb R_+,\\
        &z(0) = z_0,
    \end{array}
    \right.
\end{equation}
and the system
\begin{equation}
\label{eq:fast_ODE}
\left\{
    \begin{array}{cl}
        &y_t + y_x + y_{xxx}=0,\: (t,x)\in \mathbb R_+ \times [0,L],\\
        &y(t,0) = y(t,L) = 0, t\in \mathbb R_+,\\
        &y_x(t,L) = a z(t),\: t\in\mathbb R_+,\\
        &y(0,x) = y_0(x),\: x\in [0,L]\\
        &\varepsilon \dot z(t) = b z(t) + c y_x(t,0),\: t\in \mathbb R_+,\\
        &z(0) = z_0,
    \end{array}
    \right.
\end{equation}
where $a,b,c\in \mathbb R$ and $\varepsilon>0$. The parameter $\varepsilon$ is supposed to be small, meaning that in \eqref{eq:fast_KdV} the KdV equation is faster than the ODE, and in \eqref{eq:fast_ODE}, the ODE is faster than the KdV equation. To analyze these systems from an asymptotic stability viewpoint, we will follow techniques borrowed from the singular perturbation literature (see e.g., \cite{kokotovic1999singular,khalil2002nonlinear} for the finite-dimensional case, \cite{cerpa2019singular,tang2017stability,tang2015tikhonov} for the infinite-dimensional case). Roughly speaking, this technique proposes to decouple the full system into two approximated systems assuming that $\varepsilon$ is sufficiently small. The approximated slow system is called the reduced order system while the approximated fast one is called the boundary layer system. It is known that, in the finite-dimensional case, if both systems are asymptotically stable, then the full-system is asymptotically stable as well for sufficiently small $\varepsilon$. In general, this is no longer the case in the infinite dimensional case, as illustrated in \cite{tang2017stability,cerpa-prieur2017} for some hyperbolic equations coupled with an ODE. Therefore, the singular perturbation techniques become very challenging for infinite-dimensional systems, even in the linear case.

Regarding the partial differential part of our systems, we note that even in the case where the KdV equation is not coupled with any ODE, the asymptotic stability analysis is not trivial at all. Indeed, if $L\in \mathcal N$, with 
\begin{equation}
    \mathcal{N}:=\left\{2\pi\sqrt{\tfrac{k^2+kl+l^2}{3}}:k,l\in\mathbb N\right\},
\end{equation}
then the equilibrium point $0$ of the KdV equation becomes stable, but not attractive, while, if $L\notin \mathcal{N}$, $0$ is exponentially stable. In fact, this is linked to a lack of observability. With Neumann boundary control (i.e., a control that is acting on $y_x(t,L)$), the system is not controllable if $L\in \mathcal{N}$, as shown in \cite{rosier}. However, when looking at the nonlinear version of the KdV equation, one has better controllability results for any $L\in \mathcal N$ \cite{corcre,cerpa2007, cerpa2009boundary,coron2020small} and better stability results for some $L\in \mathcal N$
 \cite{chu2015asymptotic,coron2017local,tang2018asymptotic,nguyen2021}. In addition to these interesting results, let us mention \cite{cerpacoron,marxcerpa,tang2015stabilization,lucoron,ayadi2018exponential} which propose to apply the backstepping method on the KdV equation with various boundary control, \cite{cerpaem} where a feedback-law is designed thanks to a Gramian methodology, \cite{marx2017global} which deals with a saturated distributed control, \cite{cerpa2014,rosierzhang2009} which propose both a survey about the Kdv equation, or \cite{balogoun2021iss} where a PI controller is designed to achieve output regulation. This latter article is interesting because it is based on the forwarding method (see e.g., \cite{mazenc1996adding} for the finite-dimensional case, and \cite{terrand2019adding,marx2021forwarding,vanspranghe2022output} for some extensions to the infinite-dimensional case), which requires the existence of an ISS-Lyapunov functional (see e.g., \cite{mironchenko2020input} for an introduction on ISS). In \cite{balogoun2021iss}, an ISS-Lyapunov functional is built thanks to some strictification technique borrowed from \cite{praly2019} at the price of assuming that $L\notin \mathcal{N}$. This Lyapunov functional, which was not available before \cite{balogoun2021iss} will be crucial to analyze \eqref{eq:fast_KdV} and \eqref{eq:fast_ODE} following the classical procedure of the singular perturbation analysis. Hence, all along the paper, we will assume that 
\begin{equation}
\label{eq:assume}
L\notin \mathcal{N}.
\end{equation}

In this article, we have several contributions. First, for each coupled system \eqref{eq:fast_KdV} and \eqref{eq:fast_ODE}, we propose some conditions on the parameters $a$, $b$ and $c$ so that the exponential stability is ensured for any $\varepsilon>0$. For each of the systems, different conditions will be given, because we are going to use different Lyapunov functionals for \eqref{eq:fast_KdV} and \eqref{eq:fast_ODE}. 
Second, for each coupled system \eqref{eq:fast_KdV} and \eqref{eq:fast_ODE}, we apply the singular perturbation analysis to find he boundary layer system and the reduced order system. The stability of these subsystems will imply the stability of the original system as soon as $\varepsilon$ is small enough. Third, for each coupled system \eqref{eq:fast_KdV} and \eqref{eq:fast_ODE}, we provide an analysis of the asymptotic behavior of the solutions with respect to $\varepsilon$ by obtaining some Tikhonov theorems. To the best of our knowledge, this is the first time that a singular perturbation analysis is applied on a KdV equation, from a control viewpoint. 

This article is divided into five sections. Section \ref{sec:well-posed} is devoted to state and prove the well-posedness and stability results for \eqref{eq:fast_KdV} and \eqref{eq:fast_ODE} for any value of the parameter $\varepsilon$. In Section \ref{sec:fast_KdV} and Section \ref{sec:fast_ODE} we provide an asymptotic analysis of \eqref{eq:fast_KdV} and \eqref{eq:fast_ODE}, respectively, by applying singular perturbation analysis for small values of the parameter $\varepsilon$. Section \ref{sec_conclusion} collects some concluding remarks. Appendix \ref{sec_ISS} recalls a crucial result borrowed from \cite{balogoun2021iss} for the KdV equation subject to disturbances.


\section{Analysis for any value of $\varepsilon$}\label{sec:well-posed}

\subsection{Well-posedness}

This short section deals with the well-posedness of \eqref{eq:fast_KdV} and \eqref{eq:fast_ODE} for any parameter $a$, $b$, $c$ and $\varepsilon$. We state and prove that there exists a unique solution to both equations. Our proof relies on classical semigroup arguments. Without loss of generality, we assume that $\varepsilon=1$, because, in the well-posedness proof, this parameter does not play any role. Thus, we can deal in a unified way with both systems \eqref{eq:fast_KdV} and \eqref{eq:fast_ODE} studying  
\begin{equation}
\label{eq:single}
\left\{
    \begin{array}{cl}
        &y_t + y_x + y_{xxx}=0,\: (t,x)\in \mathbb R_+\times [0,L]\\
        &y(t,0) = y(t,L) = 0,\: t\in \mathbb R_+\\
        &y_x(t,L) = a z(t),\: t\in\mathbb R_+\\
        &y(0,x) = y_0(x),\: x\in [0,L]\\
        &\dot z(t) = b z(t) + c y_x(t,0),\: t\in \mathbb R_+\\
        &z(0) = z_0.
    \end{array}
    \right.
\end{equation}

\begin{theorem}
Let $a,b,c\in\mathbb R$. For any initial condition $(y_0,z_0)\in H^3(0,L)\times \mathbb R$ satisfying the compatibility conditions $y_0(0) = y_0(L) = 0$ and $y_0^\prime(L) = a z_0$, there exists a unique strong solution $y\in C(\mathbb R_+;H^3(0,L))\cap C^1(\mathbb R_+;L^2(0,L))$ of \eqref{eq:single}. Additionally, for any initial condition $(y_0,z_0)\in L^2(0,L)\times \mathbb R$ 
there exists a unique weak solution $y\in C(\mathbb R_+;L^2(0,L))$ to system \eqref{eq:single}.

\end{theorem}

\begin{proof}
 Applying \cite[Corollary 2.2.3]{curtain2020introduction}, we will prove the well-posedness of \eqref{eq:single}. To do so, we focus on the operator
$$
\mathcal{A}:\: D(\mathcal{A})\subset L^2(0,L)\times \mathbb R\rightarrow L^2(0,L)\times \mathbb R,
$$
where $D(\mathcal{A}):=\lbrace (y,z)\in H^3(0,L)\times \mathbb R\mid y(0) = y(L),\: y^\prime(L) = az\rbrace$ and 
\begin{equation}
\mathcal{A}\begin{pmatrix}
y \\ z
\end{pmatrix}= \begin{pmatrix}
    -y^\prime - y^{\prime\prime\prime} \\  bz + cy^\prime(0)
\end{pmatrix}.
\end{equation}
Our goal is to prove that there exists $\omega>0$ such that $\mathcal{A}-\omega \mathrm{I}_{L^2(0,L)}$ and its adjoint operator generate a strongly continuous semigroup of contractions, where $\mathrm{I}_{L^2(0,L)}$ denotes the identity operator in $L^2(0,L)$. As explained in \cite[Corollary 2.3.3]{curtain2020introduction}, and noticing moreover that $\mathcal{A}$ is a closed operator, such a condition is sufficient to prove that $\mathcal{A}$ generates a strongly continuous semigroup. Consider in $L^2(0,L)\times \mathbb R$ the scalar product
\begin{equation}
    \left\langle \begin{pmatrix}
        y_1 \\ z_1
    \end{pmatrix}, \begin{pmatrix}
        y_2 \\ z_2
    \end{pmatrix}\right\rangle = \int_0^L y_1 y_2 dx + z_1z_2.
\end{equation}
Doing some integrations by parts, one obtains, for all $(y,z)\in D(\mathcal{A})$
\begin{equation}
\left \langle \mathcal{A} \begin{pmatrix}
    y \\ z
\end{pmatrix},\begin{pmatrix}
    y \\ z
\end{pmatrix}\right\rangle = 2a^2 z^2 - 2y^\prime(0)^2 + 2b z^2 + 2cz y^\prime(0).
\end{equation}
Using Young's Lemma one obtains, for $(y,z)\in D(\mathcal{A})$
\begin{equation}
\left \langle \mathcal{A} \begin{pmatrix}
    y \\ z
\end{pmatrix},\begin{pmatrix}
    y \\ z
\end{pmatrix} \right\rangle \leq 2a^2 z^2 - 2y^\prime(0)^2 + 2b z^2 + 2 \frac{1}{\alpha} c^2 z^2 + 2\alpha y^\prime(0)^2. 
\end{equation}
If one takes $\alpha=\frac{1}{4}$, one can prove easily that there exists a positive constant $C$ such that, for all~$(y,z)\in D(\mathcal{A})$
\begin{equation}
    \left \langle \mathcal{A} \begin{pmatrix}
    y \\ z
\end{pmatrix},\begin{pmatrix}
    y \\ z
\end{pmatrix} \right\rangle \leq C(\Vert y\Vert^2_{L^2(0,L)} + z^2).
\end{equation}
Then, for any $\omega>C$, one has that $\mathcal{A}-\omega \mathrm{I}_{L^2(0,L)}$ is dissipative.

One can prove that the adjoint operator of $\mathcal{A}$, denoted by $\mathcal{A}^*$, is defined as
$$
\mathcal{A}^*:\: D(\mathcal{A}^*)\subset L^2(0,L)\times \mathbb R\rightarrow L^2(0,L)\times \mathbb R,
$$
where $D(\mathcal{A}^*):=\lbrace (y,z)\in H^3(0,L)\times \mathbb R\mid y(0) = y(L),\: y^\prime(0) = cz\rbrace$ and 
\begin{equation}
\mathcal{A}^*\begin{pmatrix}
y \\ z
\end{pmatrix}= \begin{pmatrix}
    -y^\prime - y^{\prime\prime\prime} \\  bz + ay^\prime(L)
\end{pmatrix}.
\end{equation}
Using the same scalar product than before, and performing some integrations by parts, one has, for all $(y,z)\in D(\mathcal{A}^*)$ that
\begin{equation}
\left \langle \mathcal{A}^* \begin{pmatrix}
    y \\ z
\end{pmatrix},\begin{pmatrix}
    y \\ z
\end{pmatrix}\right\rangle = 2c^2 z^2 - 2y^\prime(L)^2 + 2b z^2 + 2az y^\prime(L).
\end{equation}
Again, thanks to the Young's inequality, one can prove that
\begin{equation}
\left \langle \mathcal{A}^* \begin{pmatrix}
    y \\ z
\end{pmatrix},\begin{pmatrix}
    y \\ z
\end{pmatrix}\right\rangle \leq 2c^2 z^2 - 2y^\prime(L)^2 + 2 b z^2 + \frac{2}{\alpha} a^2 z^2 + 2\alpha y^\prime(L)^2.
\end{equation}
Setting $\alpha=\frac{1}{4}$, one can prove that there exists a positive constant $C$ such that, for all $(y,z)\in D(\mathcal{A}^*)$
\begin{equation}
\left \langle \mathcal{A}^* \begin{pmatrix}
    y \\ z
\end{pmatrix},\begin{pmatrix}
    y \\ z
\end{pmatrix}\right\rangle \leq C(\Vert y\Vert^2_{L^2(0,L)} + z^2).
\end{equation}
Then, for any $w>C$, one can prove that $\mathcal{A}^*-\omega \mathrm{I}_{L^2(0,L)}$ is dissipative. Then, applying the Lumer-Phillips Theorem \cite[Corollary 4.4, Chapter 1]{pazy}, one can deduce the result.
\end{proof}

\subsection{Stability conditions for system \eqref{eq:fast_KdV}}
Here we fix any $\varepsilon>0$ and give some conditions on $a,b$ and $c$ such that the origin is globally exponentially stable for system \eqref{eq:fast_KdV}. To do so, we have to first introduce a Lyapunov functional, inspired by \cite{balogoun2021iss}. This Lyapunov functional has been built thanks to the forwarding method, first designed for finite-dimensional systems \cite{mazenc1996adding}, and later extended to some infinite-dimensional systems \cite{balogoun2021iss,marx2021forwarding,terrand2019adding,vanspranghe2022output}. It is defined as
\begin{equation}
\label{eq:Lyapunov-coupled}
V_1(y,z) = \varepsilon W(y) + \frac{1}{2}\left(\varepsilon \int_0^L M(x) y(t,x) dx - z(t) \right)^2,
\end{equation}
where $W$ comes from \cite[Theorem 2.3]{balogoun2021iss} and we recall in the Appendix \ref{sec_ISS}. The function $M$ is the solution to the boundary value problem
\begin{equation}
\label{eq:bdv-M}
\left\{
\begin{array}{cl}
     &  M^{\prime \prime \prime}(x) + M^\prime(x) = 0,\quad x \in (0,L),\\
     & M(0) = M(L) = 0,\quad M^\prime(0) = -c,
\end{array}
\right.
\end{equation}
for which we know an explicit solution
\begin{equation}
\label{eq:explicit-M}
M(x)= 2c\frac{\sin\left(\frac{x}{2}\right)\sin\left(\frac{L-x}{2}\right)}{\sin\left(\frac{L}{2}\right)}\in C^\infty([0,L]).
\end{equation}
This function $M$ is defined through a Sylvester equation, as explained in \cite{balogoun2021iss}. Roughly speaking, the idea of the Lyapunov functional defined in \eqref{eq:Lyapunov-coupled} is to use that the fast system (i.e. the KdV equation) is already exponentially stable without coupling terms and to add a term such that $z$ converges to the $L^2$-norm of $y$ (modulo the function $M$, suitably chosen). This corresponds exactly to the forwarding method. 

As proved in \cite{balogoun2021iss}, this Lyapunov functional is equivalent to the usual norm, i.e. one has the following lemma whose proof is given for the sake of completeness.
\begin{lemma}
\label{lem:V1-norm}
There exist $\overline{\nu}_1$, $\underline{\nu}_1>0$ such that
\begin{equation}
\underline{\nu}_1(\Vert y\Vert^2_{L^2(0,L)}+|z|^2)\leq V_1(y,z) \leq \overline{\nu}_1(\Vert y\Vert^2_{L^2(0,L)} + |z|^2),
\end{equation}
with $\overline{\nu}_1 = \max\left(\varepsilon \overline c+\varepsilon^2 \Vert M\Vert^2_{L^2(0,L)},1\right)$ and $\underline{\nu}_1=\min\left(\frac{\underline{c}\varepsilon}{2}, \frac{1}{2}\frac{\underline c \varepsilon}{\varepsilon^2 \Vert M\Vert^2_{L^2(0,L)}+ \underline c \varepsilon} \right)$. Moreover, for $\varepsilon\leq 1$, one has the existence of a constant $C>0$ such that
\begin{equation}
 \varepsilon (\Vert y\Vert^2_{L^2(0,L)} + |z|^2) \leq  V_1(y,z) \leq C (\Vert y\Vert^2_{L^2(0,L)} + |z|^2).
\end{equation}
\end{lemma}

\begin{proof}
First, using Proposition \ref{prop:ISS-Lyap} in Appendix \ref{sec_ISS} and Young's Lemma we have
\begin{equation}
V_1(y,z) \leq (\varepsilon \overline{c} + \varepsilon^2 \Vert M\Vert^2_{L^2(0,L)}) \Vert y\Vert_{L^2(0,L)} + |z|^2.
\end{equation}
Second, using again Proposition \ref{prop:ISS-Lyap} and Young's Lemma we get
\begin{equation}
V_1(y,z) \geq \underline{c} \varepsilon \Vert y\Vert^2_{L^2(0,L)} + \frac{1}{2}\left(1 -\frac{1}{\alpha}\right)\varepsilon^2\int_0^L M(x)^2 y(t,x)^2 + \frac{1}{2}(1-\alpha) z(t)^2. 
\end{equation}
Choose $\alpha = \frac{\varepsilon^2\Vert M\Vert^2_{L^2(0,L)}}{\varepsilon^2\Vert M\Vert^2_{L^2(0,L)}+\underline c \varepsilon}$. Then, $1-\frac{1}{\alpha}<0$, and one has
\begin{equation}
\begin{split}
&V_1(y,z) \geq \\
& \hspace{0.5cm}\underline c\varepsilon \Vert y\Vert^2_{L^2(0,L)} - \frac{1}{2}\left(\frac{\underline c \varepsilon}{\varepsilon^2 \Vert M\Vert_{L^2(0,L)}}\right) \varepsilon^2 \Vert M\Vert^2_{L^2(0,L)} \Vert y\Vert^2_{L^2(0,L)} +\frac{1}{2}\frac{\underline c \varepsilon}{\varepsilon^2 \Vert M\Vert^2_{L^2(0,L)}+ \underline c \varepsilon} z^2.
\end{split}
\end{equation}
This concludes the proof. 
\end{proof}

We are now ready to state and prove our stability result.

\begin{proposition}
\label{prop:coupled}
For any $\varepsilon>0$, there exist positive constants $a_*$, $k_1$, $k_2$ such that, if $a<a_*$ and $b,c$ satisfy $0<k_1<-(b-ac)<k_2$, then the origin  is globally exponentially stable for system \eqref{eq:fast_KdV}.
\end{proposition}

\begin{proof}
Using Proposition \ref{prop:ISS-Lyap}, setting $d_1=0$ and $d_2=z$, the time derivative of $V$ along the strong solutions to \eqref{eq:fast_KdV} yields
\begin{multline}
    \frac{d}{dt} V(y,z) \leq - \lambda \Vert y\Vert^2_{L^2(0,L)} + \kappa_2 a^2 z(t)^2 \\
    + \left(\int_0^L M(y_x(t,x) + y_{xxx}(t,x)) - bz(t) - cy_x(t,0)\right)\left(\varepsilon \int_0^L M y(t,x) dx - z(t)\right).
\end{multline}
After some integration by parts, and using in particular that $M^\prime(L) = c$ thanks to \eqref{eq:explicit-M}, one obtains that for all strong solutions to \eqref{eq:fast_KdV} we have
\begin{multline}
\frac{d}{dt} V_1(y,z) \leq - \lambda \Vert y\Vert^2_{L^2(0,L)} + \kappa_2 a^2 z(t)^2 - (b-ac) z(t)\left(\varepsilon \int_0^L M(x) y(t,x) dx - z(t)\right) \\
\leq  - \lambda \Vert y\Vert^2_{L^2(0,L)} + \kappa_2 a^2 z(t)^2 + (b-ac) z(t)^2 - \varepsilon(b-ac) z(t) \int_0^L M(x) y(t,x) dx.
\end{multline}
Using Young's Lemma, one obtains that, for all strong solutions to \eqref{eq:fast_KdV}
\begin{equation}
\frac{d}{dt} V_1(y,z) \leq (-\lambda + \alpha \varepsilon^2 \Vert M\Vert^2_{L^2(0,L)}) \Vert y\Vert^2_{L^2(0,L)} + \left(\frac{(b-ac)^2}{\alpha} + (b-ac) + \kappa_2 a^2\right) z(t)^2.
\end{equation}
Let us choose $\alpha=\frac{\lambda}{2\Vert M\Vert^2_{L^2(0,L)} \varepsilon^2}$. One has therefore
\begin{equation}
\frac{d}{dt} V_1(y,z) \leq -\frac{\lambda}{2}\Vert y\Vert^2_{L^2(0,L)} + \left(\frac{(b-ac)^2}{\alpha} + (b-ac) + \kappa_2 a^2\right) z(t)^2.
\end{equation}
Let us consider the polynomial $\frac{X^2}{\alpha} - X + \kappa_2 a^2$. If $a^2<\frac{\alpha}{4\kappa_2}$, this polynomial admits two square roots, defined by
$$
X_1 = \frac{\alpha\left(1-\sqrt{1-\frac{4\kappa_2 a^2}{\alpha}}\right)}{2},\: X_2 = \frac{\alpha\left(1+\sqrt{1-\frac{4\kappa_2 a^2}{\alpha}}\right)}{2}.
$$
Then, if $b-ac$ satisfies
$$
X_1<-(b-ac)<X_2,
$$
then, there exists a positive constant $\mu$ such that, for all strong solutions to \eqref{eq:fast_KdV} we have
\begin{equation}
\frac{d}{dt} V_1(y,z) \leq -\mu V_1(y,z).
\end{equation}
Using Lemma \ref{lem:V1-norm} we conclude the proof.
\end{proof}

One might see this result as an extension of the one provided in \cite{balogoun2021iss} where one has $b=0$ and $c=\varepsilon=1$, which corresponds to the case where an integrator is added. In \cite{balogoun2021iss}, it is proved that, for a sufficiently small $a$, the origin of \eqref{eq:fast_KdV} (with $b=0,\: c=\varepsilon=1$) is exponentially stable. Therefore, Proposition \ref{prop:coupled} seems to follow the same line, since $a$ has to be sufficiently small.

\subsection{Stability conditions for system \eqref{eq:fast_ODE}}
In this subsection, a sufficient conditions on $a$, $b$ and $c$ will be found to ensure the stability of \eqref{eq:fast_ODE} for any $\varepsilon>0$. To do so, we use the Lyapunov functional 

\begin{equation}
\label{eq:Lyapunov-fastODE}
V_2(y,z):= -\frac{\varepsilon \kappa_2 a^2}{b} z^2 + W(y),
\end{equation}
where $W$ is the ISS-Lyapunov functional given in Proposition \ref{prop:ISS-Lyap}. The following Lemma states that this Lyapunov functional is equivalent to the usual norm.

\begin{lemma}
\label{lem:Lyapunov-fastODE-equi}
For any $b<0$, defining $\overline{\nu}_2:= \max\left(\bar c,-\frac{\varepsilon \kappa_2 a^2}{b}\right)$ and $\underline{\nu}_2:=\min\left(\underline c,-\frac{\varepsilon \kappa_2 a^2}{b}\right)$, the Lyapunov functional defined in \eqref{eq:Lyapunov-fastODE} satisfies
\begin{equation}
\underline{\nu}_2(\Vert y\Vert^2_{L^2(0,L)} + |z|^2)\leq V_2(y,z) \leq \overline{\nu}_2 (\Vert y\Vert^2_{L^2(0,L)} + |z|^2).
\end{equation}
\end{lemma}
\begin{proof}
Using Proposition \ref{prop:ISS-Lyap}, one first has
\begin{equation}
    V_2(y,z) \leq \overline c \Vert y\Vert^2_{L^2(0,L)} - \frac{\varepsilon \kappa_2 a^2}{b} \leq \overline{\nu}_2  (\Vert y\Vert^2_{L^2(0,L)} + |z|^2),
\end{equation}
where $\overline{\nu}_2=\max\left(\bar c,-\frac{\varepsilon \kappa_2 a^2}{b}\right)$. Using again Proposition \ref{prop:ISS-Lyap}, one obtain
\begin{equation}
V_2(y,z) \geq \underline{c} \Vert y\Vert^2_{L^2(0,L)} - \frac{\varepsilon \kappa_2 a^2}{b} \geq \underline{\nu}_2  (\Vert y\Vert^2_{L^2(0,L)} + |z|^2),
\end{equation}
where $\underline{\nu}_2 = \min\left(\underline c,-\frac{\varepsilon \kappa_2 a^2}{b}\right)$. This concludes the proof.\end{proof}

We have now the following result, which states that, for any $\varepsilon>0$, and under suitable conditions on $a,b,c$, the origin  is exponentially stable for system \eqref{eq:fast_ODE}. As explained later on, these conditions differ from the ones collected in Proposition \ref{prop:coupled}. 
\begin{proposition}
\label{prop:coupled2}
Let $\varepsilon>0$. If $b<0$ and $\frac{a^2c^2}{b^2} < \frac{\kappa_3}{4\kappa_2}$, then the origin  is exponentially stable for system \eqref{eq:fast_ODE}.
\end{proposition}

\begin{proof}
 Note that, due to the condition $b<0$, the Lyapunov functional defined in \eqref{eq:Lyapunov-fastODE} is equivalent to the usual norm, invoking Lemma \ref{lem:Lyapunov-fastODE-equi}. Using Proposition \ref{prop:ISS-Lyap} with $d_2=az$, its derivative along \eqref{eq:fast_ODE} yields, for all strong solutions to \eqref{eq:fast_ODE}
\begin{equation}
\frac{d}{dt} V_2(y,z) = -\lambda \Vert y\Vert^2_{L^2(0,L)} + \kappa_2 a^2 |z(t)|^2 - \kappa_3 |y_x(t,0)|^2- 2\kappa_2 a^2 z^2 - 2\frac{a^2 c}{b} \kappa_2 y_x(t,0) z(t).
\end{equation}
Using Young's Lemma, one obtains
\begin{equation}
    \frac{d}{dt} V_2(y,z) \leq -\lambda \Vert y\Vert^2_{L^2(0,L)} - \kappa_2(1-2\alpha) a^2 z^2 - \left(\kappa_3 - 2\frac{a^2 c^2 \kappa_2}{b^2\alpha}\right) |y_x(t,0)|^2
\end{equation}
Setting $\alpha =\frac{1}{4}$, one obtains:
\begin{equation}
 \frac{d}{dt} V_2(y,z) \leq -\lambda \Vert y\Vert^2_{L^2(0,L)} - \frac{\kappa_2}{2} a^2 z^2 - \left(\kappa_3 - \frac{8 a^2 c^2 \kappa_2}{b^2}\right) |y_x(t,0)|^2
\end{equation}
Then, if $\frac{a^2c^2}{b^2} < \frac{\kappa_3}{4\kappa_2}$, and using Lemma \ref{lem:Lyapunov-fastODE-equi}, the desired result holds true, concluding therefore the proof.\end{proof}

Note that the conditions given in Proposition \ref{prop:coupled} are quite different from the ones introduced in Proposition \ref{prop:coupled2}. Indeed, in contrast with Proposition \ref{prop:coupled}, Proposition \ref{prop:coupled2} assumes, with the hypothesis $b<0$, that the ODE is already exponentially stable. As it will be illustrated later on, similar conditions will appear when looking at the reduced order system and the boundary layer system.

\section{Fast KdV equation coupled with a slow ODE}
\label{sec:fast_KdV}
\subsection{Stability for small $\varepsilon$}

The singular perturbation method proposes the decoupling of the different time-scales appearing in the system in order to get some subsystems that hopefully can be studied separately in order to conclude properties of the full system. Thus, we are going to compute the subsystems, namely the reduced order system and the boundary layer system, that are approximations of the KdV equation and the ODE when $\varepsilon$ is closed to $0$. We further prove that the stability conditions for those two systems apply for the full system \eqref{eq:fast_KdV} as soon as $\varepsilon$ is small enough.

\paragraph{Reduced order system.} Finding the reduced order system needs us to suppose that $\varepsilon=0$. One has therefore to study this system
\begin{equation}
\label{eq:equi-point}
\left\{
\begin{array}{cl}
&h_x(t,x) + h_{xxx}(t,x) = 0,\quad t\in \mathbb R_+, x\in(0,L), \\
&h(t,0) = h(t,L) = 0, \quad t\in \mathbb R_+,\\
&h_x(t,L) = az(t), \quad t\in \mathbb R_+,
\end{array}
\right.
\end{equation}
which corresponds to the KdV equation given in \eqref{eq:fast_KdV} when $\varepsilon=0$. There exists an explicit solution to the latter equation given by $$h(t,x) = - 2az(t) \frac{1}{\sin\left(\frac{L}{2}\right)} \sin\left(\frac{x}{2}\right) \sin\left(\frac{L-x}{2}\right),\: \forall (t,x)\in \mathbb{R}_+\times [0,L].$$ 
One can easily check  that $h(t,0) = h(t,L) = 0$, for all $t\geq 0$. Moreover, one has, for all $(t,x)\in \mathbb R_+\times [0,L]$
$$
h_x(t,x) = -az(t)\frac{1}{\sin\left(\frac{L}{2}\right)}\cos\left(\frac{x}{2}\right)\sin\left(\frac{L-x}{2}\right) + az(t) \frac{1}{\sin\left(\frac{L}{2}\right)}\sin\left(\frac{x}{2}\right) \cos\left(\frac{L-x}{2}\right).
$$
One has $h_x(t,0) = -az(t)$ and $h_x(t,L) = az(t)$, for all $t\geq 0$. Moreover, by definition of $h$, one has $h_x(t,x) + h_{xxx}(t,x) = 0$, for all $(t,x)\in \mathbb R_+\times [0,L]$. In the following, we will use the following notation

$$
h(t,x):= -f(x)z(t),\: \text{ where } f(x):=2a\frac{1}{\sin\left(\frac{L}{2}\right)} \sin\left(\frac{x}{2}\right) \sin\left(\frac{L-x}{2}\right).
$$

Therefore, since $h_x(t,0) = -az(t)$, the reduced order system is given by
\begin{equation}
\label{eq:reduced-order}
\left\{
\begin{array}{cl}
     & \dot{\bar z}(t) = (b-ac) \bar z(t), \quad t\in \mathbb R_+,\\
     & \bar z(0) = \bar z_0.
\end{array}
\right.
\end{equation}
In consequence, if $(b-ac)<0$, then the origin of \eqref{eq:reduced-order} is exponentially stable.

\paragraph{Boundary layer system.} Consider $\tau = \frac{t}{\varepsilon}$ and $\bar y(\tau,x):= y(\tau,x) + z(\tau) f(x)$, for all $(\tau,x)\in \mathbb R_+\times [0,L]$. One can check that $$\bar y_\tau(\tau,x) = y_\tau(\tau,x) + \varepsilon\frac{d}{dt} f(x) z(t),$$ for all $(\tau,x)\in \mathbb R_+\times [0,L]$. Setting $\varepsilon=0$ yields $\bar y_\tau(\tau,x) = y_\tau(\tau,x)$. One can check also easily that $\bar y_x(\tau,x) + \bar y_{xxx}(\tau,x) = y_x(\tau,x) + y_{xxx}(\tau,x)$, for all $(\tau,x)\in \mathbb R_+ \times [0,L]$. One has also $\bar y(\tau,0) = \bar y(\tau,L) = 0$ and $\bar y_x(\tau,L) = 0$. Finally, the boundary layer is written as 
\begin{equation}
\label{eq:boundary-layer}
\left\{
\begin{array}{cl}
     & \bar y_\tau(\tau,x) + \bar y_{x}(\tau,x) + \bar y_{xxx}(\tau,x) = 0,\quad \tau\in \mathbb R_+, x\in(0,L),\\
     & \bar y(\tau,0) = \bar y(\tau,L) = 0, \quad \tau\in \mathbb R_+,\\
     & \bar y_x(\tau,L) = 0, \quad \tau\in \mathbb R_+,\\
     & \bar y(0,x) = \bar y_0(x)\quad x\in (0,L).
\end{array}
\right.
\end{equation}
Since $L\notin \mathcal{N}$, the origin  is always exponentially stable for system \eqref{eq:boundary-layer}. 

\paragraph{Full system.} Next result will say that the conditions for the reduced order system and the boundary layer system to be exponentially stable are sufficient for the full-system as soon as $\varepsilon$ is sufficiently small. It is useful to introduce the variable
 $$\tilde y(t,x) = y(t,x) + f(x)z(t),\quad \forall (t,x)\in \mathbb R_+ \times [0,L]. $$ Noticing that $\tilde y_x(t,0) = y_x(t,0) + az(t)$, its dynamics together with the one of $z$ is given by
\begin{equation}
\label{eq:sys-fastKdVthm1}
\left\{
\begin{array}{cl}
     & \varepsilon \tilde y_t + \tilde y_x + \tilde y_{xxx} = -  \varepsilon ((b-ac)z(t) + c\tilde y_x(t,0))f(x)  \\
     & \tilde y(t,0)=\tilde y(t,L) = 0\\
     & \tilde y_x(t,L) = 0\\
     & \tilde y(0,x) = \tilde y_0(x)\\
     &\dot z = (b-ac) z(t) + c\tilde y_x(t,0)\\
     & z(0)=z_0.
\end{array}
\right.
\end{equation}
We can now state and prove the next result. 

\begin{theorem}
\label{thm:fast-KdV1}
For any $a,b,c\in\mathbb R$ such that $(b-ac)<0$, there exists $\varepsilon^*>0$ such that, for every $\varepsilon\in (0,\varepsilon^*)$  the origin is exponentially stable for system \eqref{eq:fast_KdV}.
\end{theorem}

\begin{proof}
 We consider the Lyapunov functional \eqref{eq:Lyapunov-coupled}. Applying Proposition \ref{prop:ISS-Lyap} with $d_1(t,x)=-  \varepsilon ((b-ac)z(t) + c\tilde y_x(t,0))f(x)$ and $d_2(t) = 0$, one obtains that all strong solutions to \eqref{eq:sys-fastKdVthm1} satisfy
\begin{multline}
\frac{d}{dt} V_1(\tilde y,z) \leq -\lambda \Vert \tilde y\Vert^2_{L^2(0,L)} + \varepsilon^2 \kappa_1 \Vert f\Vert^2_{L^2(0,L)} ((b-ac)z(t)+c\tilde y_x(t,0))^2 - \kappa_3 \tilde y_x(t,0)^2\\
\hspace{-0.5cm}+ \left(-(b-ac)z(t) + K \varepsilon ((b-ac)z(t) + c\tilde y_x(t,0))\right)\cdot \left(\varepsilon \int_0^L M(x) \tilde y(t,x) dx - z(t)\right),
\end{multline}
where $K:= \int_0^L M(x) f(x) dx$. Using Young's Lemma several times 
one obtains, that for all strong solutions to \eqref{eq:sys-fastKdVthm1}
\begin{multline}
\frac{d}{dt} V_1(\tilde y,z) \leq \left(-\lambda + \alpha_1 \Vert M\Vert^2_{L^2(0,L)} + \alpha_2 \varepsilon^2 \Vert M\Vert^2_{L^2(0,L)}\right)\Vert \tilde y\Vert^2_{L^2(0,L)} \\
\hspace{-0.5cm}+ \left((b-ac) + (b-ac)^2\left(2\varepsilon^2 \kappa_1\Vert f\Vert^2_{L^2(0,L)} + \frac{\varepsilon^2 }{\alpha_1} + 2K^2\varepsilon^2 \left(\frac{1}{\alpha_2} + \frac{1}{\alpha_3}\right)\right)+\alpha_3\right) z(t)^2\\
 + \left(2\varepsilon^2 \kappa_1 \Vert f\Vert^2_{L^2(0,L)} +K^2c^2\varepsilon^2\left(\frac{2}{\alpha_2}+\frac{2}{\alpha_3}\right)-\kappa_3\right) \tilde y_x(t,0)^2. 
\end{multline}

Selecting $\alpha_3=-\frac{-(b-ac)}{2}$, setting $\alpha_4 =\frac{1}{\alpha_2} + \frac{1}{\alpha_3}$, and choosing $\varepsilon$ and $\alpha_1$ and $\alpha_2$ satisfying
\begin{equation}
\varepsilon^2 <\min\left(\frac{\kappa_3}{2\kappa_1 \Vert f\Vert^2_{L^2(0,L)} + 2K^2c^2\alpha_4},\frac{1}{2(ac-b)\left(2\kappa_1\Vert f\Vert^2_{L^2(0,L)} + \frac{1}{\alpha_1} + 2K^2\varepsilon^2 \alpha_4\right)}\right)
\end{equation}
and
\begin{equation}
-\lambda + \alpha_1 \Vert M\Vert^2_{L^2(0,L)} + \alpha_2 \varepsilon^2 \Vert M\Vert^2_{L^2(0,L)}<0,
\end{equation}
one obtains that there exists $\mu>0$ such that
\begin{equation}
    V_1(\tilde y,z)\leq e^{-\mu t} V(\tilde y_0,z_0),\quad \forall t\geq 0.
\end{equation}
Using Lemma \ref{lem:V1-norm}, one deduces the desired result.
\end{proof}

\subsection{Tikhonov theorem}
The most relevant part of the singular perturbation method is to use the obtained subsystems in order to approximate the dynamics of the full system. This section is devoted to this more precise analysis of the asymptotic behavior of the solutions with respect to the variable $\varepsilon$. To do so, we will follow the Tikhonov strategy that has been used for instance in \cite{tang2015tikhonov,cerpa2019singular} for partial differential equations.  We introduce the error solutions
\begin{equation}
\hat{z}(t) = z(t) - \bar z(t)
\end{equation}
and 
\begin{equation}
\hat{y}(t,x) = y(t,x) + f(x)\bar z(t) - \bar y\left(\frac{t}{\varepsilon},x\right). 
\end{equation}
Using the solutions of \eqref{eq:fast_KdV}, \eqref{eq:reduced-order}, \eqref{eq:equi-point} and \eqref{eq:boundary-layer}. One can verify that
$$
\dot{\hat{z}}(t) = b z(t) + cy_x(t,0) - (b-ac) \bar z(t).
$$
Noticing that $\hat{y}_x(t,x) = y_x(t,x) + f^\prime(x) \bar z(t) - \bar y_x\left(\frac{t}{\varepsilon},x\right)$, one has $$\hat{y}_x(t,0) = y_x(t,0) - a \bar z(t) - \bar y_x\left(\frac{t}{\varepsilon},0\right)$$ and because $\dot{\hat{z}}(t) = b (\bar z(t) - z(t)) + c \bar y_x\left(\frac{t}{\varepsilon},0\right)$, we get

\begin{equation}
\label{eq:EDO-Tikhonov}
\dot{\hat{z}}(t) = b\hat{z}(t) + c \hat y_x(t,0) + c \bar y_x\left(\frac{t}{\varepsilon},0\right). 
\end{equation}
Moreover, 
\begin{equation}
\varepsilon \hat y_t(t,x) = \varepsilon y_t(t,x) + \varepsilon f(x) (b-ac) \bar z(t) - \bar y_t\left(\frac{t}{\varepsilon},x\right).
\end{equation}
Using the dynamics of $y$ (given in \eqref{eq:fast_KdV}) and the one of $\bar y$ (given in \eqref{eq:boundary-layer}), one obtains
\begin{equation}
\varepsilon \hat y_t(t,x) = -y_x - y_{xxx} + \bar y_{x}\left(\frac{t}{\varepsilon},x\right) + \bar y_{xxx}\left(\frac{t}{\varepsilon},x\right) + \varepsilon f(x) (b-ac) z(t).    
\end{equation}
Note that 
\begin{equation*}
    \hat y_x + \hat y_{xxx} = y_{x} + y_{xxx} - \bar y_{x}\left(\frac{t}{\varepsilon},x\right) - \bar y_{xxx}\left(\frac{t}{\varepsilon},x\right) + \bar z(t)(f^\prime(x) + f^{\prime\prime \prime}(x)).
\end{equation*}
Recall that $h(t,x) = -\bar z(t) f(x)$ and that $h$ solves \eqref{eq:equi-point}, i.e.  $\bar z(t)(f^\prime(x) + f^{\prime\prime \prime}(x))= 0$. Hence, one has
\begin{equation}
\varepsilon \hat y_t(t,x) = -\hat y_x - \hat y_{xxx} + \varepsilon f(x) (b-ac) \bar z(t). 
\end{equation}
Using the boundary conditions given in \eqref{eq:fast_KdV}, \eqref{eq:boundary-layer} and \eqref{eq:equi-point}, one has
$$
\hat y(t,0) = \hat y(t,L) = 0,\: \forall t\geq 0.
$$
Having in mind that $\hat y_x(t,L) = y_x(t,L) + f^\prime(L) \bar z(t) - \bar y_x\left(\frac{t}{\varepsilon},L\right)= az(t) - a \bar z(t)=a\hat z(t)$, one  can write the system

\begin{equation}
\label{eq:Tikhonov-KdV}
\left\{
\begin{array}{cl}
     & \varepsilon \hat y_t + \hat y_x + \hat y_{xxx} = \varepsilon f(x)(b-ac) \bar z(t),\quad t\in \mathbb{R}_+, x\in(0,L),  \\
     & \hat y(t,0) = \hat y(t,L) = 0, \quad t\in \mathbb{R}_+,\\
     & \hat y_x(t,L) = a \hat z(t), \quad t\in \mathbb{R}_+,\\
     & \hat y(0,x) = y_0(x)-\bar y_0(x)+f(x)\bar z(0), \quad x\in (0,L),\\
     &\dot{\hat z} = b \hat z(t) + c\hat y_x(t,0)+c\bar y_x(t/\varepsilon,0), \quad t\in \mathbb{R}_+,\\
     & \hat z(0)=z_0-\bar z_0.
\end{array}
\right.
\end{equation}

We are now in position to state our first Tikhonov theorem. 
\begin{theorem}
There exist positive constants $a_*$, $k_1$, $k_2$ and $\varepsilon^*$ such that if $a<a_*$, $b,c$ satisfy $0<k_1<-(b-ac)<k_2$ and $\varepsilon<\varepsilon^*$, then 
for any initial conditions $(y_0,z_0),(\bar y_0,\bar z_0)\in L^2(0,L)\times \mathbb R$ such that 
\begin{equation*}
\Vert y_0 - \bar y_0+ f z_0 \Vert_{L^2(0,L)} + |z_0-\bar z_0|= O(\varepsilon^{\frac{3}{2}}),\quad 
\quad |\bar z_0|=O(\varepsilon^{\frac{1}{2}}),\quad \Vert \bar y_0\Vert_{L^2(0,L)}=O(\varepsilon^{\frac{3}{2}}),
\end{equation*}
we have that the solutions of \eqref{eq:fast_KdV} satisfy for some $\mu > 0$ that
\begin{equation}
\Vert y(t,\cdot) - \bar y(t/\varepsilon,\cdot)+ f(\cdot)z(t) \Vert_{L^2(0,L)} +
|z(t)-\bar z(t)|=O(\varepsilon) e^{-\mu t}.
\end{equation}

\end{theorem}

\begin{proof}
Let us consider the Lyapunov functional \eqref{eq:Lyapunov-coupled}. Its derivative along strong solutions to \eqref{eq:Tikhonov-KdV} yields
\begin{multline*}
\frac{d}{dt} V_1(\hat y,\hat z) \leq - \lambda \Vert \hat y\Vert^2_{L^2(0,L)} + \kappa_1 \varepsilon^2 \Vert f\Vert^2_{L^2(0,L)}(b-ac)^2 \bar z(t)^2 + \kappa_2 a^2 \hat z(t)^2\\
+ \left(K\varepsilon(b-ac)\bar z(t) - (b-ac) \hat z(t) + c\bar y_x(t,0)\right)\left(\varepsilon\int_0^L M(x) \hat y(t,x) dx - \hat z(t)\right),
\end{multline*}
where $K:=\int_0^L f(x) M(x) dx$. Using Young's Lemma several times,
one obtains for all strong solutions of \eqref{eq:Tikhonov-KdV} that
\begin{multline*}
\frac{d}{dt} V_1(\hat y,\hat z)\leq  (-\lambda + (\alpha_1 \varepsilon^2 + \alpha_3 + \alpha_4)\Vert M\Vert^2_{L^2(0,L)}) \Vert \hat y\Vert^2_{L^2(0,L)}\\
+ \left(\kappa_2 a^2 + (b-ac) + \alpha_2 + \alpha_5 + (b-ac)^2 \frac{\varepsilon^2}{\alpha_3}\right) \hat z(t)^2\\
 + K^2\varepsilon^2 (b-ac)^2 \left(\frac{1}{\alpha_1} + \frac{1}{\alpha_2}\right) \bar z(t)^2
 + c^2 \left(\frac{\varepsilon^2}{\alpha_4} + \frac{1}{\alpha_5}\right) \bar y_x(t,0)^2.
\end{multline*}
One selects $\alpha_1,\alpha_3$ and $\alpha_4$ such that
$$
-\lambda + (\alpha_1 \varepsilon^2 + \alpha_3 + \alpha_4)\Vert M\Vert^2_{L^2(0,L)} <0
$$
and $\alpha_2$ and $\alpha_5$ such that
$$
\alpha_2 + \alpha_5 = -\frac{b-ac}{2}.
$$
Then, setting $\mu_1:=\lambda - (\alpha_1 \varepsilon^2 + \alpha_3 + \alpha_4)\Vert M\Vert^2_{L^2(0,L)}< \lambda $, one has
\begin{multline*}
\frac{d}{dt} V(\hat y,\hat z)\leq  -\mu_1 \Vert \hat y\Vert^2_{L^2(0,L)}
+ \left(\kappa_2 a^2 + \frac{(b-ac)}{2} + (b-ac)^2 \frac{\varepsilon^2}{\alpha_3}\right) \hat z(t)^2\\
 + K^2\varepsilon^2 (b-ac)^2 \left(\frac{1}{\alpha_1} + \frac{1}{\alpha_2}\right) \bar z(t)^2
 + c^2 \left(\frac{\varepsilon^2}{\alpha_4} + \frac{1}{\alpha_5}\right) \bar y_x(t,0)^2.
\end{multline*}
Let us now consider the polynomial $P(X) = \kappa_2a^2 - \frac{1}{2} X + \frac{\varepsilon^2}{\alpha_3}X^2$. If $a$ is such that $a^2<\frac{\alpha_3}{16\varepsilon^2 \kappa_2}$, $P(X)$ admits two square roots
\begin{equation}
X_1 = \alpha_3\frac{1 - \sqrt{1-\frac{16\varepsilon^2 \kappa_2 a^2}{\alpha_3}}}{4\varepsilon^2},\: X_2 =  \alpha_3\frac{1 + \sqrt{1-\frac{16\varepsilon^2 \kappa_2 a^2}{\alpha_3}}}{4\varepsilon^2}.
\end{equation}
Replacing $X$ by $ac-b$, one has $P(ac-b)<0$ if 
$$
X_1<ac-b<X_2.
$$
Then, there exists $\mu_2>0$ such that, for all strong solutions to \eqref{eq:Tikhonov-KdV}
\begin{multline}
\label{theo:Tikhonov-before-integral}
\frac{d}{dt} V(\hat y,\hat z)\leq  -\mu_1 \Vert \hat y\Vert^2_{L^2(0,L)} - \mu_2 \hat z(t)^2
+ K^2\varepsilon^2 (b-ac)^2 \left(\frac{1}{\alpha_1} + \frac{1}{\alpha_2}\right) \bar z(t)^2\\
 + c^2 \left(\frac{\varepsilon^2}{\alpha_4} + \frac{1}{\alpha_5}\right) \bar y_x\left(\frac{t}{\varepsilon},0\right)^2.
\end{multline}
Using Proposition \ref{prop:ISS-Lyap}, the Gr\"onwall's Lemma and setting $\mu_3:=\min\left(\frac{\mu_1}{\overline{\nu_1}},\frac{\mu_2}{\overline{\nu_1}}\right)$, one obtains for all $t\geq 0$ that
\begin{equation}
\label{eq:inequalityV_1-Tikhonov}
V_1(\hat y,\hat z) \leq e^{-\mu_3 t} V(\hat y_0,\hat z_0) + O(\varepsilon^2) \int_0^t e^{-\mu_3 (t-s)} |\bar z(s)|^2 ds + O(1) \int_0^t e^{-\mu_3(t-s)} \bar y_x^2\left(\frac{t}{\varepsilon},0\right)
\end{equation}
Let us estimate the integrals appearing in \eqref{eq:inequalityV_1-Tikhonov}. Using \eqref{eq:reduced-order}, one has
\begin{equation}
\label{eq:inequality-barz}
\bar z(t)^2\leq e^{-(b-ac) t} \bar z_0^2. 
\end{equation}
Moreover, using the Lyapunov functional given in Proposition \ref{prop:ISS-Lyap}, one obtains for all strong solutions to \eqref{eq:boundary-layer} that
\begin{equation*}
\frac{d}{d\tau} W(\bar y) \leq - \frac{\lambda}{\overline c} W(\bar y) - \kappa_3 \bar y_x(\tau,0),
\end{equation*}
with $\tau=\frac{t}{\varepsilon}$. Notice that
\begin{equation}
   \mu_3 \leq \frac{\mu_1}{\overline \nu_1} \leq \frac{\mu_1}{\varepsilon \overline c + \varepsilon^2 \Vert M\Vert^2_{L^2(0,L)}}\leq \frac{\mu_1}{\varepsilon\overline c} \leq \frac{\lambda}{\varepsilon \bar c}
\end{equation}
where we have used the definition of $\nu_1$ given in Lemma \ref{lem:V1-norm} and the definition of $\mu_1$ and $\mu_3$. Then, one has for all strong solutions to \eqref{eq:boundary-layer} that
\begin{equation*}
\frac{d}{d\tau} W(\bar y) \leq - \varepsilon \mu_3 W(\bar y) - \kappa_3 \bar y_x(\tau,0).
\end{equation*}

Hence, using the Gr\"onwall's Lemma again one obtains for all $\tau\geq 0$ that
\begin{equation*}
W(\bar y) \leq e^{-\varepsilon\mu_3 \tau} W(\bar y_0) - \int_0^\tau e^{-\varepsilon \mu_3 (\tau-s)} \bar y_x(\tau,0)^2 d\tau.
\end{equation*}
One can conclude that
\begin{equation}
\label{eq:hidden-regularity1}
\int_0^{\tau} e^{-\varepsilon \mu_3 (\tau-s)} y_x(s,0)^2 ds \leq \bar c e^{-\varepsilon\mu_3 \tau} \Vert \bar y_0\Vert^2_{L^2(0,L)}.
\end{equation}
Setting $\bar s = \varepsilon s$ and using the definition of $\tau$, one obtains
\begin{equation}
\label{eq:hidden-regularity2}
\int_0^{\frac{t}{\varepsilon}} e^{-\mu_3 (t-\bar s)} \bar y_x\left(\frac{\bar s}{\varepsilon},0\right)^2 ds \leq \bar c e^{-\mu_3 t} \Vert \bar y_0\Vert^2_{L^2(0,L)}.
\end{equation}
Then, using \eqref{eq:inequality-barz} and \eqref{eq:hidden-regularity2}, and noticing that, since $\varepsilon\leq 1$, one has $\frac{t}{\varepsilon}\geq t$, one obtains finally that for all $t\geq 0$
\begin{equation}
V_1(\hat y,\hat z) \leq O(1) e^{-\mu_3 t} V_1(\hat y_0,\hat z_0) + O(\varepsilon^2) e^{-(b-ac) t} |\bar z_0|^2 + O(1) e^{-\mu_3 t} \Vert \bar y_0\Vert^2_{L^2(0,L)}.
\end{equation}
Taking $\mu_4=\min((b-ac),\mu_3)$, and using the smallness condition on the initial conditions, one obtains that
\begin{equation}
V_1(\hat y,\hat z) \leq e^{-\mu_4 t} O(\varepsilon^3).
\end{equation}
Using Lemma \ref{lem:V1-norm}, one has $V_1(\hat y,\hat z)> O(\varepsilon)(\Vert \hat y\Vert_{L^2(0,L)} + |\hat z|)^2$, concluding thus the proof.
\end{proof}

\section{Fast ODE coupled with a slow KdV equation}
\label{sec:fast_ODE}

\subsection{Stability for small $\varepsilon$}

Following the steps in the singular perturbation method, we are going to compute the reduced order system and the boundary layer system for \eqref{eq:fast_ODE}. The exponential stability conditions will be drastically different, which explains why we used a different Lyapunov functional. In addition to this different Lyapunov functional, these conditions will hold at the price of considering strong solutions to \eqref{eq:fast_ODE}.

\paragraph{Reduced order system.} Setting $\varepsilon=0$, one obtains that $z(t) = -\frac{c}{b}y_x(t,0)$. The reduced order system, whose state is denoted by $\bar y$, satisfies

\begin{equation}
\label{eq:reduced_order1}
\left\{
    \begin{array}{cl}
        &\bar y_t + \bar y_x + \bar y_{xxx}=0,\: (t,x)\in\mathbb R_+\times [0,L],\\
        &\bar y(t,0) = \bar y(t,L) = 0,\: t\in\mathbb R_+\\
        &\bar y_x(t,L) = -\frac{ac}{b} \bar y_x(t,0),\: t\in \mathbb R_+\\
        &\bar y(0,x) = \bar y_0(x),\: x\in [0,L].
    \end{array}
    \right.
\end{equation}

Using the Lyapunov functional given in Proposition \ref{prop:ISS-Lyap} with $d_2(t) = -\frac{ac}{b}y_x(t,0)$, one has

\begin{equation}
    \frac{d}{dt} W(\bar y) \leq -\lambda W + \left(\kappa_2\frac{a^2c^2}{b^2}-\kappa_3\right) y_x(t,0)^2.
\end{equation}
Hence, if $\frac{a^2c^2}{b^2}<\frac{\kappa_3}{\kappa_2}$, then the origin of \eqref{eq:reduced_order1} is ensured to be exponentially stable. Note that this conditions looks like the one given in \cite{zhang1994boundary}. However, in this latter paper, one requires that $\left|-\frac{ac}{b}\right|<1$. This is surely associated to the fact that the Lyapunov approach is more conservative than the one followed in \cite{zhang1994boundary}. 

\paragraph{Boundary layer system.} Consider $\tau = \frac{t}{\varepsilon}$ and $\bar z(\tau) = z(\tau) + \frac{c}{b} y_x(t,0)$. One has $\dot{\bar z}(\tau) = \dot z(\tau) + \varepsilon \frac{d}{dt} \frac{c}{b} y_{xt}(t,0)$. With $\varepsilon=0$, one obtains that $$\dot{\bar z}(\tau) = \dot z(\tau) = b z(\tau) + c y_x(\tau,0)= b(z(\tau) + \frac{c}{b} y_x(\tau,0)) = b \bar z(\tau).$$ Then, the boundary layer system is defined by

\begin{equation}
    \label{eq:boundary-layer1}
    \dot{\bar z}(\tau) = b \bar z(\tau),
\end{equation}
which means that its origin is exponentially stable if $b<0$.

\paragraph{Full system.} 

Consider $\tilde z(t) = z(t) + \frac{c}{b} y_x(t,0)$. Then, $y$ solves the following equation
\begin{equation}
\label{eq:reduced_order1}
\left\{
    \begin{array}{cl}
        & y_t + y_x + y_{xxx}=0,\: (t,x)\in\mathbb R_+\times [0,L],\\
        & y(t,0) = y(t,L) = 0,\: t\in\mathbb R_+,\\
        & y_x(t,L) = a\tilde z(t) -\frac{ac}{b} y_x(t,0),\: t\in\mathbb R_+\\
        & y(0,x) = y_0(x),\: x\in [0,L].
    \end{array}
    \right.
\end{equation}
One has $\varepsilon \dot{\tilde z} = \varepsilon \dot z + \varepsilon\frac{c}{b} y_{tx}(t,0) = b z(t) + \frac{c}{b} y_{x}(t,0) + \varepsilon \frac{c}{b} y_{tx}(t,0)$, i.e.,
\begin{equation}
\label{eq:fast_ODE-difference}
\varepsilon \dot{\tilde z}(t) = b \tilde z(t) + \varepsilon \frac{c}{b}y_{tx}(t,0).
\end{equation}

Since the dynamics of $\tilde z$ introduces the time-derivative of $y_x(t,0)$, one needs more regularity on $y$. We consider therefore $v=\varepsilon y_t$. The use of the parameter $\varepsilon$ in the variable $v$ is due to the fact that, in \eqref{eq:reduced_order1} it appears the state $\tilde z$. The dynamics of $v$ is given by
\begin{equation}
\label{eq:reduced_order_regular}
\left\{
    \begin{array}{cl}
        & v_t + v_x + v_{xxx}=0,\: (t,x)\in\mathbb R_+\times [0,L],\\
        & v(t,0) = v(t,L) = 0, \: t\in \mathbb R_+,\\
        & v_x(t,L) = ab \tilde z(t),\: t\in\mathbb R_+\\
        & v(0,x) = -\varepsilon y^\prime_0-\varepsilon y^{\prime\prime\prime}_0,\: x\in [0,L],\\
        &\varepsilon \dot {\tilde z}(t) = b \tilde z(t) + \varepsilon\frac{c}{b} v_x(t,0),\: t\in \mathbb R_+\\
        &\tilde z(0) = z_0+\frac{c}{b}y^\prime_0.
    \end{array}
    \right.
\end{equation}

Well-posedness of \eqref{eq:reduced_order_regular} is given by our results in Section \ref{sec:well-posed}.  Because we work in $L^2$-regularity for $v$ and $H^3$-regularity for $y$, some compatibility conditions appear on the intitial data. We are in position to state the following result.
\begin{theorem}
\label{th:singular2}
For any $a,b,c\in\mathbb R$ such that $\frac{a^2c^2}{b^2}<\frac{\kappa_3}{\kappa_2}$, where $\kappa_2$ and $\kappa_3$ are defined in Proposition \ref{prop:ISS-Lyap}, there exists $\varepsilon^*>0$ such that for any $\varepsilon\in (0,\varepsilon^*)$, the origin of \eqref{eq:fast_ODE} is exponentially stable for any initial conditions $(y_0,z_0)\in H^3(0,L)\times \mathbb R$ such that $$y_0(0)=y_0(L) = 0,\quad y_0^\prime(L) = ab(z_0+\frac c b y_0^\prime(0)).$$ 
\end{theorem}
\begin{proof}
To prove this result, we consider the following Lyapunov functional
\begin{equation}
\label{eq:V_3}
V_3(v,\tilde z) = W(v) - \varepsilon \kappa_2 a^2 b   \tilde z^2,
\end{equation}
where $W$ is the Lyapunov functional given in Proposition \ref{prop:ISS-Lyap}.
Using the same proof as in Lemma \ref{lem:Lyapunov-fastODE-equi}, one has that
for any $b<0$, we can define $\overline{\nu}_3=\max\left(\overline{c}, -\varepsilon\kappa_2^2a^2 b\right)$ and $\underline{\nu}_3=\min\left(\underline c,-\varepsilon\kappa_2 a^2 b\right)$, where $\kappa_2>0$ comes from Proposition \ref{prop:ISS-Lyap} such that the Lyapunov functional \eqref{eq:V_3} satisfies
\begin{equation}\label{ex-lema}
  \underline{\nu}_3 (\Vert v\Vert_{L^2(0,L)} + |\tilde z|^2) \leq  V_3(v,\tilde z) \leq \overline{\nu}_3(\Vert v\Vert_{L^2(0,L)} + |\tilde z|^2).
\end{equation}

 Time derivative of \eqref{eq:V_3} along the strong solutions to \eqref{eq:reduced_order_regular} yields
\begin{equation}
\begin{split}
\frac{d}{dt}V_3(v,z) \leq & - \lambda \Vert v\Vert^2_{L^2(0,L)} + \kappa_2 a^2b^2 \tilde z^2 - \kappa_3 v_x(0)^2 - 2\kappa_2 a^2 b^2 \tilde z^2 + 2\kappa_2 a^2\left(b\tilde z \varepsilon \frac{c}{b}v_x(t,0)\right).
\end{split}
\end{equation}
Using Young's Lemma, one gets
\begin{equation}
\frac{d}{dt}V_3(v,z) \leq - \lambda \Vert v\Vert^2_{L^2(0,L)} - \kappa_2 a^2 b^2 \tilde z^2 - \kappa_3 v_x(0)^2 + 2\alpha \kappa_2 a^2 b^2 \tilde z^2 + \frac{2}{\varepsilon \alpha} \frac{a^2c^2\kappa_2}{b^2} v_x(t,0)^2
\end{equation}
and setting $\alpha = \frac{1}{4}$ one obtains
\begin{equation}
\frac{d}{dt}V_3(v,z) \leq - \lambda \Vert v\Vert^2_{L^2(0,L)} - \kappa_2 \frac{a^2 b^2}{2} \tilde z^2 + \left(8\varepsilon^2\kappa_2 \frac{a^2c^2}{b^2}-\kappa_3\right) v_x(0)^2.
\end{equation}
If one takes $\varepsilon<\frac{1}{8}$, $b<0$ and $\frac{a^2c^2}{b^2}<\frac{\kappa_3}{\kappa_2}$, then one obtains the desired result using \eqref{ex-lema}. \end{proof}

\subsection{Tikhonov theorem}

This subsection is devoted to the asymptotic analysis of \eqref{eq:fast_ODE} with respect to $\varepsilon$. As before, such an analysis requires to consider strong solutions to \eqref{eq:fast_ODE}.  Let us introduce the two variables

\begin{equation}
\hat z(t) = z(t) + \frac{c}{b} y_x(t,0) - \bar z\left(\frac{t}{\varepsilon}\right),\: \hat y(t) = y(t) - \bar y(t).
\end{equation}
One can check that
\begin{equation}
\label{eq:EDO-Tikhonov2}
\varepsilon \dot{\hat z}(t) = b\hat z(t) + \frac{\varepsilon c}{b}\left(\hat y_{tx}(t,0) + \bar y_{tx}(t,0)\right)
\end{equation}
and
\begin{equation}
\left\{
\begin{array}{cl}
     & \hat y_t + \hat y_x + \hat y_{xxx}=0, \: (t,x)\in \mathbb R_+\times [0,L], \\
     & \hat y(t,0) = \hat y(t,L) = 0,\: t\in \mathbb R_+,\\
     & \hat y_x(t,L) = a\left(\hat z(t) + \bar z\left(\frac{t}{\varepsilon}\right)\right),\: t\in\mathbb R_+,\\
     & \hat y(0,x) = \hat y_0(x),\: x\in [0,L].
\end{array}
\right.
\end{equation}
Since \eqref{eq:EDO-Tikhonov2} introduces the time-derivative of $y_x(t,0)$, let us consider $\hat v = \varepsilon\hat y_t$, where $\varepsilon$ is introduced because the boundary condition makes appear $z$ and $\bar z$. Its dynamics satisfies the following system
\begin{equation}
\label{eq:EDP-Tikhonov2}
\left\{
\begin{array}{cl}
     & \hat v_t + \hat v_x + \hat v_{xxx}=0,\: (t,x)\in\mathbb R_+\times [0,L],\\
     & \hat v(t,0) = \hat v(t,L) = 0,\: t\in \mathbb R_+,\\
     & \hat v_x(t,L) = ab \hat z(t) + \varepsilon\frac{ac}{b}\left(v_x(t,0) + \bar y_{tx}(t,0))\right) + ab \bar z\left(\frac{t}{\varepsilon}\right),\: t\in\mathbb R_+,\\
     & \hat v(0,x) = \hat v_0(x),\: x\in [0,L].
\end{array}
\right.
\end{equation}
Let us consider the Lyapunov functional
\begin{equation}
\label{eq:V_4}
    V_4(\hat v,\hat z) = W(\hat v) - 3\varepsilon \kappa_2  a^2 b |\hat z|^2.
\end{equation}
We can prove, as before, the following. For any $b<0$, defining $\overline{\nu}_4:=\max(\overline c,-3\varepsilon\kappa_2 a^2b)$ and $\underline{\nu}_4:=\min(\underline c,-3\varepsilon \kappa_2 a^2b)$, where $\kappa_2$ comes from Proposition \ref{prop:ISS-Lyap}, the Lyapunov functional \eqref{eq:V_4} satisfies
\begin{equation}\label{lem:V_4}
    \underline{\nu}_4(\Vert \hat v\Vert^2_{L^2(0,L)} + |\hat z\Vert|^2)\leq V_4(\hat v,\hat z)\leq \overline{\nu}_4(\Vert \hat v\Vert^2_{L^2(0,L)} + |\hat z|^2).
\end{equation}
Moreover, if $\varepsilon\leq 1$, then there existe $C>0$ such that
\begin{equation}
   \varepsilon (\Vert \hat v\Vert^2_{L^2(0,L)} + |\hat z|^2)\leq  V_4(\hat v,\hat z) \leq C(\Vert \hat v\Vert^2_{L^2(0,L)} + |\hat z|^2).
\end{equation}
We are now in position to state and prove our Tikhonov theorem for \eqref{eq:fast_ODE}.
\begin{theorem}
There exist $\varepsilon^*$ and $\mu>0$ such that for any $\varepsilon\in (0,\varepsilon^*)$, for any $b<0$, for any $a,c\in\mathbb R$ such that $\frac{a^2c^2}{b^2}<\frac{\kappa_3}{44\kappa_2 \varepsilon^2}$, where $\kappa_2$ and $\kappa_3$ come from Proposition \ref{prop:ISS-Lyap}, and for any initial conditions $(y_0,z_0)\in H^3(0,L)\times\mathbb R$ satisfying the compatibility conditions $$y_0(0) = y_0(L) = 0,\quad y^\prime(L) = az_0$$ and the smallness conditions
\begin{equation}
\begin{split}
&\Vert y_0-\bar y_0\Vert_{H^3(0,L)} + \left|z_0 + \frac{c}{b} y^\prime_0(0) - \bar{z}_0\right|= O(\varepsilon^{5/2}) \\
&|\bar z_0|=O(\varepsilon^{5/2}),\quad \Vert \bar{y}_0\Vert_{H^3(0,L)} = O(\varepsilon^{3/2}),
\end{split}
\end{equation}
then for all $t\geq 0$
\begin{equation}
\Vert y(t,\cdot)-\bar y(t,\cdot)\Vert_{H^3(0,L)} + \left|z(t) + \frac{c}{b} y_x(t,0) - \bar{z}(t/\varepsilon)\right|=O(\varepsilon)e^{-\mu t}. 
\end{equation}
\end{theorem}

\begin{proof}
Using Proposition \ref{prop:ISS-Lyap} with $d_1=0$, $d_2 = ab \hat z(t) + \varepsilon\frac{ac}{b}\left(v_x(t,0) + \bar y_{tx}(t,0))\right) + ab \bar z\left(\frac{t}{\varepsilon}\right)$, the time derivative of $V_4$ along strong solutions to \eqref{eq:EDO-Tikhonov2}-\eqref{eq:EDP-Tikhonov2} yields
\begin{multline}
    \frac{d}{dt} V_4(\hat v,\hat z) \leq - \lambda \Vert \hat v\Vert^2_{L^2(0,L)} + \kappa_2 \left(ab \hat z(t) + \varepsilon \frac{ac}{b}\left(v_x(t,0) + \bar y_{tx}(t,0)\right) + ab \bar z\left(\frac{t}{\varepsilon}\right)\right)^2\\
    - 3\kappa_2 a^2 b^2 \hat z^2 - \frac{3\varepsilon \kappa_2 a^2}{c} \frac{1}{b}(\hat v_x(t,0) + \bar y_{tx}(t,0))b\hat z(t).
\end{multline}
Using Young's Lemma we get
\begin{multline}
\kappa_2 \left(ab \hat z(t) + \frac{ac}{b}\left(v_x(t,0) + \bar y_{tx}(t,0)\right) + ab \bar z\left(\frac{t}{\varepsilon}\right)\right)^2\\ \leq  2\kappa_2 a^2b^2 \hat z(t)^2 
+ 4 \kappa_2 \varepsilon^2 \frac{a^2c^2}{b^2}\left(v_x(t,0) + \bar y_{tx}(t,0)\right)^2 
+ 4 \kappa_2 a^2b^2 \bar z\left(\frac{t}{\varepsilon}\right)^2\\
\leq  2\kappa_2 a^2b^2 \hat z(t)^2 
+ 8 \kappa_2 \varepsilon^2 \frac{a^2c^2}{b^2}\left(\hat v_x(t,0)^2 + \bar y_{tx}^2(t,0)\right) 
 + 4 \kappa_2 a^2 b^2 \bar z\left(\frac{t}{\varepsilon}\right)^2
\end{multline}
and
\begin{equation}
\frac{3\kappa_2 a^2}{c} \frac{1}{b}(\hat v_x(t,0) + \bar y_{tx}(t,0))b\hat z(t)\leq \frac{6\varepsilon^2 \kappa_2 a^2c^2}{\alpha  b^2}(\hat v_x(t,0)^2 + \bar y^2_{tx}(t,0)) + 3 a^2b^2 \alpha \kappa_2 \hat z(t)^2. 
\end{equation}
Let us take $\alpha = \frac{1}{6}$ and gather all the above inequalities to get
\begin{equation}
\begin{split}
 \frac{d}{dt} V_4(\hat v,\hat z) \leq & - \lambda \Vert \hat v\Vert^2_{L^2(0,L)} - \frac{1}{2} \kappa_2 a^2 b^2 \hat z^2 + \left(\frac{44\varepsilon^2 \kappa_2 a^2c^2}{b^2} - \kappa_3\right) v_x(t,0)^2 \\
 & + 4 \kappa_2 a^2 b^2 \bar z\left(\frac{t}{\varepsilon}\right)^2 + \frac{44\varepsilon^2 \kappa_2 a^2c^2}{b^2} \bar y_{tx}(t,0)^2.
 \end{split}
 \end{equation}
 Choosing $\frac{a^2c^2}{b^2} \leq \frac{\kappa_3}{44\kappa_2\varepsilon^2}$ as in the statement of the theorem, one obtains for all solutions to \eqref{eq:EDP-Tikhonov2}-\eqref{eq:EDO-Tikhonov2}
\begin{equation}
\frac{d}{dt} V_4(\hat v,\hat z) \leq- \lambda \Vert \hat v\Vert^2_{L^2(0,L)} - \frac{1}{2} \kappa_2 a^2 b^2  \hat z^2 + 4 \kappa_2 a^2 b^2 \bar z\left(\frac{t}{\varepsilon}\right)^2 + \frac{44\varepsilon^2 \kappa_2 a^2c^2}{b^2} \bar y_{tx}^2(t,0). 
\end{equation}
Denoting by $\mu:=\min\left(\frac{\lambda}{\overline c},-\frac{1}{6\varepsilon b}\right)$, where $\overline c$ comes from Proposition \ref{prop:ISS-Lyap}, and using the Gr\"onwall's Lemma, one obtains, for all $t\geq 0$
\begin{equation}
V_4(\hat v,\hat z) \leq e^{-\mu t} V_4(\hat v_0,\hat z_0) + \int_0^t e^{-\mu(t-s)} \left(O(1) \bar z\left(\frac{s}{\varepsilon}\right)^2 + O(\varepsilon^2) \bar y_{tx}^2(s,0)\right)ds.
\end{equation}
On one hand, one can prove that, for all $t\geq 0$
$$
\int_0^t e^{-\mu(t-s)}  \left|\bar{z}\left(\frac{s}{\varepsilon}\right)\right|^2 ds \leq O(1) e^{b\frac{t}{\varepsilon}} |\bar z_0|^2.
$$
On the other hand, consider the variable $\bar v\left(t,x\right)= \varepsilon \bar y_t(t,x)$. It satisfies the following KdV equation
\begin{equation}
\label{eq:reduced_order_regular1}
\left\{
\begin{array}{cl}
  & \bar v_t + \bar v_x + \bar v_{xxx} = 0,\: (t,x)\in\mathbb R_+\times [0,L], \\
   &  \bar v(t,0) = \bar v(t,L) = 0,\: t\in\mathbb R_+,\\
   & \bar v_x(t,L) = -\frac{ac}{b} \bar v_x(t,0),\: t\in\mathbb R_+,\\
   & \bar v_x(0,x) = \bar v_0(x),\: x\in [0,L].
\end{array}
\right.
\end{equation}

Using the ISS-Lyapunov functional given in Proposition \ref{prop:ISS-Lyap} with $d_2(t) = -\frac{ac}{b}\bar v_x(t,0)$ along the solutions to \eqref{eq:reduced_order_regular1}, one obtains
\begin{equation}
\frac{d}{dt} W(\bar v)\leq -\lambda \Vert \bar v\Vert^2_{L^2(0,L)} + \left(\frac{a^2c^2}{b^2}\kappa_2-\kappa_3\right) |\bar v_x(t,0)|^2.
\end{equation}
Under the condition on $a$, $c$ and $b$, there exists $\kappa_4>0$ such that
\begin{equation}
\frac{d}{dt} W(\bar v) \leq - \lambda \Vert \bar v\Vert^2_{L^2(0,L)} - \kappa_4 \bar |v_x(t,0)|^2.
\end{equation}
Using Proposition \ref{prop:ISS-Lyap} and since $\mu\leq \frac{\lambda}{\overline c}$, one has, for all strong solution to \eqref{eq:reduced_order_regular1}
\begin{equation}
\frac{d}{dt} W(\bar v) \leq - \mu W(\bar v) - \kappa_4 v_x(t,0)^2.
\end{equation}
Thanks to Gr\"onwall's Lemma, one gets
\begin{equation}
\int_0^t e^{-\mu(t-s)} |v_x(s,0)|^2 ds \leq \frac{1}{\kappa_4} e^{-\mu t } W(\bar v_0).
\end{equation}
Therefore, using the smallness conditions given in the statement of the theorem, one has for all $t\geq 0$
\begin{equation}
\begin{split}
V_4(\hat v,\hat z) \leq &e^{-\mu t} V_4(\hat v_0,\hat z_0) + O(1) e^{bt} |\bar z_0|^2 + O(\varepsilon^2) e^{-\mu t}\|\bar y_0\|^2_{H^3}\\
\leq & e^{-\mu_1 t} O(\varepsilon^5),
\end{split}
\end{equation}
where $\mu_1=\min (\mu,-b)$.

Due to inequality \eqref{lem:V_4}, one has $V_4(\hat v,\hat z)\geq O(\varepsilon)(\Vert \hat v\Vert^2_{L^2(0,L)} + |\hat z|^2)$. Moreover, recall that $\hat v = \varepsilon y_t = - \varepsilon (y_x + y_{xxx})$. Hence, $V_4(\hat v,\hat z)\geq O(\varepsilon^3) (\Vert \hat y\Vert^2_{H^3(0,L)} + |\hat z|^2)$. Then, one can deduce the desired result concluding the proof. \end{proof}

\section{Conclusion}
\label{sec_conclusion}

In this paper, we have provided a singular perturbation analysis for two coupled systems composed by a KdV equation and an ODE. In particular, we have proved that, the conditions for the reduced order system and the boundary layer system to be exponentially system also work for the full-system for $\varepsilon$ small enough. Different Lyapunov functionals have been introduced for the cases where the KdV equation is faster or the ODE is faster. For both cases, the ISS Lyapunov functional built in \cite{balogoun2021iss} has been instrumental. It is also worth mentioning that, when the ODE is faster than the KdV equation, the perturbation analysis can be performed only for sufficiently smooth solutions.


\appendix 

\section{ISS-Lyapunov functional}

\label{sec_ISS}

This appendix recalls a crucial result provided in \cite{balogoun2021iss}, which proposes the construction of a ISS-Lyapunov functional for the KdV equation. This result will be instrumental all along this paper.  To introduce it, let us focus on the following disturbed KdV equation
\begin{equation}
\label{eq:KdV_disturbed}
\left\{
\begin{array}{cl}
     &y_t + y_x + y_{xxx}=d_1(t,x),\quad (t,x)\in \mathbb R_+ \times (0,L),\\
     &y(t,0) = y(t,L) = 0,\quad t\in\mathbb R_+,\\
     &y_x(t,L) = d_2(t),\quad t\in \mathbb R_+,\\
     &y(0,x) = y_0(x),\quad x\in [0,L],
\end{array}
\right.
\end{equation}
where $d_1\in L^2(0,T;L^2(0,L))$ and $d_2\in L^2(0,T)$, for any $T\geq 0$. The well-posedness of \eqref{eq:KdV_disturbed} can be obtained using semigroup theory in standard way for strong or mild solutions, depending on the regularity of the data. According to \cite[Theorem 2.3]{balogoun2021iss}, one has the following result.
\begin{proposition}
\label{prop:ISS-Lyap}
There exists an ISS-Lyapunov functional for \eqref{eq:KdV_disturbed}, i.e. there exists a function $W:L^2(0,L) \rightarrow \mathbb R$ and positive constants $\lambda,\kappa_1,\kappa_2,\kappa_3,\overline c,\underline c$ such that
\begin{equation}
\underline c \Vert y\Vert^2_{L^2(0,L)}\leq W(y)\leq \overline c \Vert y\Vert^2_{L^2(0,L)}
\end{equation}
and the derivative of $W$ along the solutions to \eqref{eq:KdV_disturbed} satisfies
\begin{equation}
\frac{d}{dt} W(y) \leq - \lambda \Vert y\Vert^2_{L^2(0,L)} + \kappa_1 \Vert d_1(t,\cdot)\Vert^2_{L^2(0,L)} + \kappa_2 |d_2(t)|^2 -\kappa_3 |y_x(t,0)|^2.
\end{equation}
\end{proposition}
Note that the term $-\kappa_3 |y_x(t,0)|^2$ does not appear in \cite[Theorem 2.3]{balogoun2021iss}, but following the proof in that paper, one can prove that such a term exists. It will be useful in our context.




 \bibliographystyle{plain}
\bibliography{bibsm}
\end{document}